\documentclass[12pt]{amsart}
\usepackage{latexsym}
\usepackage{amsmath}
\usepackage{epsfig}
\usepackage{epic}
\usepackage{amssymb,amsthm}
\usepackage[parfill]{parskip}
\usepackage{enumerate}
\usepackage{changepage}
\usepackage{tikz} 
\usetikzlibrary{calc}
\usepackage{comment}
   
\newtheorem{Theorem}{Theorem}[section]
\newtheorem{Lemma}[Theorem]{Lemma}
\newtheorem{Corollary}[Theorem]{Corollary}

\theoremstyle{definition}
\newtheorem{definition}[Theorem]{Definition}

\theoremstyle{remark}

\definecolor{dkgreen}{rgb}{0,.7,.1}
\definecolor{purple}{rgb}{.5,0,.6}
\newcommand{\startalign}{\setcounter{equation}{0}\begin{align}}

\begin{document}

\sloppy

\title{Recognising Graphic and Matroidal Connectivity Functions}

\author{Nathan Bowler \and Susan Jowett}
\address{School of Mathematics Statistics and Operations Research,
Victoria University of Wellington}
\email{swoppit@gmail.com}
\thanks{Susan Jowett's research was supported by an MSc scholarship from Victoria University of Wellington.}


\date{}

\begin{abstract}
A {\em connectivity function} on a set $E$ is a function $\lambda:2^E\rightarrow \mathbb R$
such that $\lambda(\emptyset)=0$, that $\lambda(X)=\lambda(E-X)$ for all $X\subseteq E$, and
that $\lambda(X\cap Y)+\lambda(X\cup Y)\leq \lambda(X)+\lambda(Y)$ for all 
$X,Y \subseteq E$. Graphs, matroids and, more generally, polymatroids have associated
connectivity functions. In this paper we give a method for identifying when a connectivity function comes from a graph. This method uses no more than a polynomial number of evaluations of the connectivity function. In contrast, we show that the problem of identifying when a connectivity function comes from a matroid cannot be solved in polynomial time. We also show that the problem of identifying when a connectivity function is not that of a matroid cannot be solved in polynomial time.
\end{abstract}

\maketitle

\section{Introduction}
A {\em connectivity function} on a set $E$ is a function $\lambda:2^E\rightarrow \mathbb R$
such that $\lambda(\emptyset)=0$, that $\lambda(X)=\lambda(E-X)$ for all $X\subseteq E$, and
that $\lambda(X\cap Y)+\lambda(X\cup Y)\leq \lambda(X)+\lambda(Y)$ for all 
$X,Y \subseteq E$. A number of mathematical structures such as graphs, matroids and polymatroids have associated connectivity functions.

A particularly natural class of connectivity functions are {\em graphic connectivity functions}, that is connectivity functions that are the connectivity functions of graphs. Our first theorem is as follows:

\begin{Theorem}\label{main theorem} There is a polynomial $p$ such that, given an arbitrary connectivity function $\lambda$, we are able to establish whether or not $\lambda$ is the connectivity function of a graph with $n$ edges in at most $p(n)$ evaluations of the connectivity function.
\end{Theorem}

Recognition problems like this are well studied in matroid theory, for example Seymour proved in \cite{recognising graphic matroids} that we can, in a polynomial number of rank evaluations, recognise when a matroid is graphic. He later showed in \cite{Seymour} that we can recognise, again in a polynomial number of rank evaluations, when a binary matroid is regular, that is, when it can be represented over every field. In \cite{Truemper}, Truemper gives a method for recognising when a matroid is regular in a polynomial number of rank evaluations.

The result of Seymour on recognising graphic matroids is in many ways analogous to our result for recognising graphic connectivity functions. Seymour proves that we can, in a polynomial number of evaluations of the rank function, recognise when a matroid is graphic, whereas we prove that we can, in a polynomial number of evaluations of the connectivity function, recognise when a connectivity function is graphic. Broadly speaking the structure of the two proofs is similar, although the details are very different. The proof of Seymour's result relies on building a graph, from the fixed matroid $M$, whose cycle matroid would be $M$, were $M$ the cycle matroid of a graph. This is where most of the difficulty lies, as we must  build, from the original matroid, $M$, a binary matroid $M'$ such that $M=M'$ if, and only if, $M$ is binary. We then check whether or not $M'$ is graphic, and, if it is not, then $M$ cannot be graphic. We then find a graph $G$ such that $M'= M(G)$. Checking whether or nor $M'$ is graphic relies on a result of Tutte \cite{Tutte}, which gives a method for determining when a binary matroid is graphic.
Checking that $M=M(G)$ is then fairly straightforward, all that is required is to check all complete stars of $G$, and, if these are all cocircuits of $M$, then $M=M(G)$.

To prove that we can, in a polynomial number of evaluations of the connectivity function, recognise whether a connectivity function $\lambda$ is a graphic connectivity function we first build a graph that would have $\lambda$ as its connectivity function were $\lambda$ graphic. This is fairly straightforward, as we can easily find the edge adjacencies such a graph would have to have were it to have connectivity function $\lambda$. From there it is not particularly difficult to build the graph with those edge adjacencies.  The second part of the proof involves checking that the connectivity function of the graph we just built is equal to $\lambda$. We must check considerably more sets than just the stars of the graph, although it turns out that these sets can be described very succinctly. 

Not all recognition problems can be solved in polynomial time. In \cite{recognising graphic matroids} Seymour proved that we cannot recognise binary matroids in a polynomial number or rank evaluations. More precisely, he showed that the number of rank evaluations needed to guarantee that a matroid on $n$ elements is binary grows superpolynomially in $n$; indeed, if $n$ is even then there is a binary matroid $M$ on $n$ elements such that for any set $A$ of less than $2^{n/2}$ subsets of the ground set there is a non-binary matroid whose rank function agrees with that of $M$ on those sets. In other words, the problem cannot be solved using only a polynomial number of calls to the rank oracle.

However, it can be shown that a matroid is non-binary in a polynomial number of calls to the rank oracle, and indeed we have the following very strong statement: for any non-binary matroid $M$ there is a set of 16 subsets of the ground set $E$ of $M$ such that no other matroid on $E$ which agrees with $M$ about the ranks of all those subsets can be binary. Indeed, since $M$ is not binary it must have a minor $N = M/P\backslash Q$ isomorphic to $U_{2,4}$. The rank function of this minor is given by $r_N(X) = r_M(P \cup X) - r_M(P)$. Thus for any other matroid $M'$ on $E$ whose rank function agrees with that of $M$ on all sets of the form $P \cup X$ with $X \subseteq E - P - Q$ we also have $M'/P\backslash Q \cong U_{2,4}$, and so $M'$ cannot be binary.

Our second main result is that the problem of identifying when a connectivity function comes from a matroid cannot generally be solved in a polynomial number of evaluations of the connectivity function. Similarly the problem of identifying when a connectivity function is not matroidal cannot generally be solved in a polynomial number of evaluations of the connectivity function. Thus the boundary between matroidal and non-matroidal connectivity functions is hard to describe, and in particular there is no characterisation corresponding to the characterisation of binary matroids by excluded minors. Like Seymour, we use spikes to generate the matroids used in the argument. However, a similar argument for the same result could also be given using sparse paving matroids instead.

Section 2 of this paper gives preliminary results that will be useful to the reader throughout the remainder of the paper. Section 3 gives details on how to build a graph from a graphic connectivity function. Section 4 proves that given a connectivity function and the results from Section 3, a connectivity function can be recognised as graphic in polynomial time. In Section 5  we show that the problem of identifying when a connectivity function comes from a matroid cannot be solved in polynomial time. We also show that the problem of identifying when a connectivity function is not that of a matroid cannot be solved in polynomial time.

\section{Preliminaries}
In this section we introduce connectivity functions, focusing particularly on graphic connectivity functions.
\begin{definition} 
Consider a set function $f$ on $E$. We say that $f$ is {\em normalised} if $f(\emptyset)=0$, that $f$ is {\em symmetric} if $f(X)=f(E-X)$ for all $X\subseteq E$, and that $f$ is {\em submodular} if $f(X)+f(Y)\geq f(X\cup Y)+f(X\cap Y)$ for all $X,Y\subseteq E$.
A normalised integer-valued set function $f$ on $E$ is {\em unitary} if $f(\{e\})\leq 1$ for all $e\in E$.
\end{definition}

\begin{definition}
A set function $\lambda:2^{|E|}\to \mathbb R$ is a {\em connectivity function} if the following hold:
\begin{enumerate}[$i$)]
\item $\lambda$ is normalised
\item $\lambda$ is symmetric
\item $\lambda$ is submodular.
\end{enumerate}
When $\lambda:2^{|E|}\to\mathbb R$, we say that $\lambda$ is {\em based} on the set $E$.
\end{definition}

Let $G$ be a graph with edge set $E$ and let $X\subseteq E$. We use $V(X)$ to denote the collection of vertices of $G$ that are incident with some edge in $X$.
We define the connectivity function of a graph as follows:

\begin{definition}\label{g conn fun} Let $G$ be a graph with vertex set $V$ and edge set $E$. The {\em connectivity function} of $G$, denoted $\gamma_G$, is defined by 
\[\gamma_G(X)=|V(X)|+|V(E-X)|-|V(E)|,\] for all $X\subseteq E$
\end{definition}

When it is clear from the context that the graph we are talking about is the graph $G$, we shall use $\gamma$ instead of $\gamma_G$.

It is easy to see that for $X\subseteq E(G)$, the connectivity of $X$, that is $\gamma(X)$, is equal to the number of vertices that $X$ and $E-X$ have in common. More formally, defining the {\em boundary} $\delta(X)$ of $X$ to be $V(X) \cap V(E-X)$, we have $\gamma(X) = |\delta(X)|$. 

The reader familiar with matroids should note that the connectivity function described above captures what is known as the {\em vertex connectivity} of the graph. It is not the connectivity function of the cycle matroid of the graph, although the two connectivity functions do have some similarities.

A proof that a graphic connectivity function is indeed a connectivity function can be found in \cite{Songbao's Thesis}.

From Definition~\ref{g conn fun} we see that the connectivity function of a graph, $G$, is based on the edge set, $E(G)$. From this it is clear that the presence of isolated vertices does not affect the connectivity function. Therefore we shall assume that our graphs do not contain isolated vertices.

\begin{definition}\label{m conn fun} Let $M$ be a matroid with groundset $E$ and rank function $r$. The {\em connectivity function} of $M$, denoted $\mu_M$, is defined by 
\[\mu_M(X)=r(X)+r(E-M)-r(M)\] for all $X\subseteq E$.
\end{definition}

If $M$ is a matroid on groundset $E$ with rank function $r$ then we use $r^*$ to denote the rank function in the dual matroid, $M^*$, and $r^*(X)=r(E-X)+|X|-r(M)$ for any $X\subseteq E-X$. It follows immediately that $\mu_M(X)=r(X)+r^*(X)-|X|$.

\section{Building a Graph from a Graphic Connectivity Function} \label{buildgraph}
We first give a method for finding the edge adjacencies of a graph from its connectivity function. We then discuss identically building the graph given its edge adjacencies.

Throughout this section, when we refer to a graph, $G$, we shall assume that $G$ has edge set $E$ and connectivity function $\gamma$, unless stated otherwise.

\begin{Lemma}\label{adjacencies}
Edges $e$ and $f$ of $G$ are adjacent if and only if $\gamma(\{e,f\}) < \gamma(\{e\}) + \gamma(\{f\})$.
\end{Lemma}
\begin{proof}
Suppose first of all that $e$ and $f$ are not adjacent. Then any vertex $x$ in $\delta(\{e\})$ is incident with $e$ and some other edge, which cannot be $f$ since $e$ and $f$ are not adjacent. So $x$ is also in $\delta(\{e, f\})$. Furthermore $\delta(\{e\})$ and $\delta(\{f\})$ are clearly disjoint. Thus $\gamma(\{e,f\}) \geq |\delta(\{e\}) \cup \delta(\{f\})| = \gamma(\{e\}) + \gamma(\{f\})$.

Now suppose instead that $e$ and $f$ are adjacent. Then any element of the boundary of $\{e,f\}$ must be in the boundary of $e$ or of $f$, and at least one vertex is in both boundaries. Thus $\gamma(\{e,f\}) \leq |\delta(\{e\}) \cup \delta(\{f\})| < \gamma(\{e\}) + \gamma(\{f\})$.
\end{proof}

Lemma~\ref{adjacencies} enables us to identify, for every pair of edges $a$ and $b$, whether or not $a$ and $b$ are adjacent by evaluating $\gamma(\{a\})$, $\gamma(\{b\})$ and $\gamma(\{a,b\})$.

Whitney proved in \cite{Whitney} that a connected graph can, under most circumstances, be built up to isomorphism from its edge adjacencies. There exist many papers, for example \cite{Zelinka} and \cite{line graph}, which give methods for building graphs from the edge adjacencies, or equivalently from their line graphs, but these methods generally only guarantee that the graph is built up to isomorphism (although often they do almost always build the graph up to identity). A method that builds the graph up to edge labelling, where possible, from the edge adjacencies can be found in \cite{Thesis} and is based on \cite{Zelinka}. In future, we shall refer to building a graph up to identity when we mean up to edge labelling. In some cases it is not  possible to build the graph up to identity from the edge adjacencies; for example $K_4- e$ cannot be built up to identity from the edge adjacencies. However, the connectivity function provides more information than just the edge adjacencies, for example we can get information about 3-element sets from the connectivity function, and this sometimes enables us to build the graph up to identity from the connectivity function when we are not able to from the adjacencies alone. In fact, the only graph we cannot build up to identity given the connectivity function is $K_4$ (which can be identified but not built up to identity).

Using these methods we can not only reconstruct the graph from the connectivity function, we can do so with only polynomially many evaluations of that function.  Of course there are some connectivity functions for which the adjacency information gleaned above is not consistent with any graph, but this too can be checked with polynomially many evaluations. This reduces our problem to the following, which we address in the next section: given a graph $G$ with edge set $E$ and a connectivity function $\lambda$ on $E$, can we check whether $\lambda = \gamma_G$ with only polynomially many evaluations of $\lambda$?

\section{Comparing the Connectivity Functions}
For this section we fix a graph $G$ and a connectivity function $\lambda$. We denote the connectivity function of $G$ by $\gamma$. For a vertex $v$ of $G$ we let $S_v$ denote the set of edges of $E$ that are incident with $v$.
\begin{definition} Let $S$ be a set and $S'\subseteq S$. We say that $S'$ is {\em controlled} if one of the following holds:
\begin{enumerate}[$i)$]
\item $S'=S$,
\item $S'=\emptyset$, 
\item $|S'|=1$.
\end{enumerate}
Let $e=uv$ be an edge of $G=(V,E)$. We say that a set $Y\subseteq E$ is {\em$e$-controlled} if it is the union of three controlled subsets, one from each of $S_x-\{e\}$, $S_y-\{e\}$ and $\{e\}$
\end{definition}
Our aim in the rest of this section is to show that if for every edge $e$ we have $\lambda(Y) = \gamma(Y)$ for all $e$-controlled sets then $\lambda = \gamma$. Since there are only polynomially many $e$-controlled sets, this will then imply that only polynomially many evaluations of $\lambda$ are needed to check whether $\lambda = \gamma$. So for the remainder of this section we shall assume that $\lambda(Y)=\gamma(Y)$ for all $e$-controlled sets $Y$.

\begin{Lemma}For any $X\subseteq E$ and any $e\in E - X$ we have $$\lambda(X\cup\{e\})-\lambda(X)\leq \gamma(X\cup \{e\})-\gamma(X)\,.$$
\begin{proof}
Let $Y$ be a minimal subset of $X$ such that for each endvertex $v$ of $e$ the following hold:
\begin{enumerate}[{$i$)}]
\item If $X$ contains an one or more edges incident with $v$ then $Y$ contains exactly one edge incident with $v$, and
\item If $X\cup \{e\}$ contains all edges incident with $v$ then so does $Y\cup \{e\}$
\end{enumerate}
Note that these conditions tell us that $Y$ and $Y\cup \{e\}$ are both $e$-controlled. Also note that an endvertex of $e$ is in the boundary of $X$ if, and only if, it is in the boundary of $Y$ and is in the the boundary of $X\cup \{e\}$ if, and only if, it is in the boundary of $Y\cup\{e\}$. On the other hand, every other vertex is in the boundary of $X$ if, and only if, it is in the boundary of $X\cup \{e\}$ and is in the boundary of $Y$ if, and only if, it is in the boundary of $Y\cup \{e\}$.

These observations imply that $\gamma(X)+\gamma(Y\cup\{e\})=\gamma(X\cup\{e\})+\gamma(Y)$.
Therefore :
\begin{align*}\lambda(X\cup \{e\})-\lambda(X)&\leq \lambda(Y\cup\{e\})-\lambda(Y)\\
&=\gamma(Y\cup\{e\})-\gamma(Y)\\
&=\gamma(X\cup \{e\})-\gamma(X)
\end{align*}
where the first inequality holds by submodularity and the first equality holds by the fact $Y$ and $Y\cup\{e\}$ are $e$-controlled.
\end{proof}
\end{Lemma}

\begin{Lemma}For any $X\subseteq E$ and any $e\in E- X$ we have $$\lambda(X\cup\{e\})-\lambda(X)= \gamma(X\cup \{e\})-\gamma(X)\,.$$
\begin{proof}
By the previous lemma $\lambda(X\cup\{e\})-\lambda(X)\leq\gamma(X\cup\{e\})-\gamma(X)$, so it remains to prove the inequality in the opposite direction. Let $X'=E-(X\cup\{e\})$.
\begin{align*}
\lambda(X\cup\{e\}) - \lambda(X)&=-(\lambda(X'\cup \{e\})-\lambda(X))\\
&\geq -(\gamma(X'\cup\{e\})-\gamma(X'))\\
&=\gamma(X\cup\{e\}) - \gamma(X)
\end{align*}
Where the first and last lines follow by symmetry of connectivity functions and the second line follows from the previous lemma applied to $X'$.
\end{proof}
\end{Lemma}

The required result then follows immediately by induction on $|X|$. That is, we have proved the following theorem:
\begin{Theorem} Let $\gamma$ be the connectivity function of a graph and let $\lambda$ be a connectivity function with the property that $\lambda(Y)=\gamma(Y)$ for all $Y$ that are e-controlled for some $e$, then $\lambda=\gamma$.
\end{Theorem}
Combining this with the results of Section \ref{buildgraph} we have a proof of Theorem~\ref{main theorem}.

\section{Matroidal Connectivity Functions}

We have seen that we can, in a polynomial number of evaluations of the connectivity function, tell if a connectivity function is graphic. We now ask the same question for matroids. In this section we use spikes, a class of matroids that provide counterexamples to many natural conjectures, to show that matroidal connectivity functions cannot be recognised in a polynomial number of evaluations of the connectivity function, nor can we recognise when a connectivity function is not that of a matroid in a polynomial number of evaluations of the connectiviy function.

We fix disjoint sets $L_i = \{x_i, y_i\}$ for each positive integer $i$, which we call {\em legs}. We denote the union of the first $n$ legs by $E_n$.

\begin{definition}
Let $n$ be an integer greater than 2. A matroid $M$ with the following properties is a {\em rank-n spike} with legs $L_1, L_2, \ldots L_n$:
\begin{enumerate}
\item $E(M) =E_n$.
\item For all $k$ in $\{1,\ldots,n-1\}$, the union of any $k$ legs of $M$ has rank $k+1$.
\item $E_n$ has rank $n$.
\end{enumerate}
\end{definition}

 The next result is taken from \cite{Oxley}:\\
\begin{Theorem}\label{circuits of a spike} Let $M$ be a matroid with ground set $E_n$ and let $\mathcal C$ be the set of circuits of $M$. Then $M$ is a rank-$n$ spike with legs $L_1, L_2, \ldots L_n$ if and only if $\mathcal C$ is equal to $\mathcal C_1\cup\mathcal C_2\cup\mathcal C_3$ where $\mathcal C_1=\{\{x_i,y_i,x_j,y_j\}:1\leq i<j\leq n\}$, $\mathcal C_2$ is a, possibly empty, subset of $\{\{z_1,\ldots,z_n\}:z_i\in\{x_i,y_i\}\}$ such that no two members of $\mathcal C_2$ differ in exactly one element, and $\mathcal C_3$ is the collection of all $(n+1)$-element subsets of $E(M)$ that contain no member of $\mathcal C_1\cup\mathcal C_2$. 
\end{Theorem}
When talking about spikes we shall use $z_i$ to describe a single element of $\{x_i,y_i\}$, and we shall refer to $\{z_1,\ldots,z_n\}$ as a {\em transversal} of a rank-$n$ spike. We call the set of such transversals $\mathcal T_n$.

The statement above can be seen as saying that spikes correspond to independent sets in the hypercube, in a sense which we now make precise. Let $H_n$ be the graph with vertex set $\mathcal T_n$ and with an edge joining 2 elements precisely when they differ in exactly one element. $H_n$ is isomorphic to the usual $n$-dimensional hypercube. If $I$ is an independent set in $H_n$ then the construction above with $\mathcal C_2 := I$ gives a rank-$n$ spike $S(I)$ and every rank-$n$ spike arises in this way. 

The rank functions of such spikes are very easy to calculate. For a subset $X$ of $E$, we define $l(X)$ to be the number of legs of the spike which $X$ meets. Then it is straightforward to check that the rank of $X$ in $S(I)$ is given as follows:
\begin{itemize}
\item If $X$ doesn't include any leg and is disjoint from some leg then it has rank $l(X) = |X|$.
\item If $X$ includes some leg and is disjoint from some other leg then it has rank $l(X) + 1$.
\item If $X$ includes some leg and meets all legs of the spike then it has rank $l(X) = n$.
\item if $X$ doesn't include any leg but meets all legs of the spike then it is a transversal. In this case it has rank $n-1$ if $X \in I$ and $n$ otherwise. 
\end{itemize}

In particular, only the ranks of transversals depend on $I$. We let $r_n$ be the function from $2^{E_n} - \mathcal T_n$ to $\mathbb N$ given by the restriction of the rank function of any rank-$n$ spike to this set. Similarly we define $\lambda_n$ to be the function from $2^{E_n} - \mathcal T_n$ to $\mathbb N$ sending $X$ to $r_n(X) + r_n({E_n} - X) - n$. Thus $\lambda_n$ is given by the restriction of the connectivity function of any rank-$n$ spike to $2^{E_n} - \mathcal T_n$. On $\mathcal T_n$, the connectivity function of a spike $S(I)$ is given by $\lambda_{S(I)}(X) = n - |I \cap \{X, {E_n} - X\}|$. We say that a function $\lambda \colon 2^{E_n} \to \mathbb N$ is {\em spiky} if it is symmetric, extends $\lambda_n$, takes values in the range $\{n-2, n-1, n\}$ on $\mathcal T_n$, and satisfies $\lambda(X) + \lambda(Y) \geq 2n-2$ for any transversals $X$ and $Y$ which differ in just one element. 

\begin{Lemma}
Any spiky function $\lambda \colon 2^{E_n} \to \mathbb N$ is a connectivity function.
\end{Lemma}
\begin{proof}
$\lambda$ is normalised since it extends $\lambda_n$ and is symmetric by definition, so it suffices to show that it is submodular. So let $X, Y \subseteq E_n$. We must show that $\lambda(X) + \lambda(Y) \geq \lambda(X \cup Y) + \lambda(X \cap Y)$. If $X \subseteq Y$ or $Y \subseteq X$ then this is clear, so we may assume that this is not the case. There are now three cases, according to $|\{X, Y\} \cap \mathcal T_n|$.

If neither $X$ nor $Y$ is in $\mathcal T_n$ then we have
\begin{eqnarray*}
\lambda(X) + \lambda(Y) &=& \mu_{S(\emptyset)}(X) + \mu_{S(\emptyset)}(Y) \\&\geq& \mu_{S(\emptyset)}(X \cup Y) + \mu_{S(\emptyset)}(X \cap Y) \\&\geq& \lambda(X \cup Y) + \lambda(X \cap Y)\, .
\end{eqnarray*}

If just one of $X$ or $Y$, say $X$, is in $\mathcal T_n$ then we have
\begin{eqnarray*}
\lambda(X) + \lambda(Y) &\geq & n-2 + \lambda_n(Y) \\&=& \mu_{S(\{X, E_n - X\})}(X) + \mu_{S(\{X, E_n - X\})}(Y) \\&\geq& \mu_{S(\{X, E_n - X\})}(X \cup Y) + \mu_{S(\{X, E_n - X\})}(X \cap Y) \\&\geq& \lambda(X \cup Y) + \lambda(X \cap Y)\, .
\end{eqnarray*}

Finally, if both $X$ and $Y$ are in $\mathcal T_n$ then if they differ in just one point we have $\lambda(X) + \lambda(Y) \geq 2n-2 = \lambda(X\cap Y) + \lambda(X \cup Y)$ and otherwise we have
\begin{eqnarray*}
\lambda(X) + \lambda(Y) &\geq& (n-2) + (n-2) \\&\geq& |E_n - (X \cup Y)| + |X \cap Y| \\&=& \lambda(X \cup Y) + \lambda(X \cap Y)\, .
\end{eqnarray*}
\end{proof}

\begin{Lemma}
If a spiky function $\lambda$ is the connectivity function of a matroid $M$ then that matroid is a spike with legs the sets $L_i = \{x_i, y_i\}$.
\end{Lemma}
\begin{proof}
We need to show that the union of any $k$ legs has rank $k+1$ for all $k\in\{1,\ldots,n-1\}$, and the rank of the union of $n$ legs is equal to $n$. First we shall look at the rank of $k$ legs for $k<n$, and without loss of generality we may take those legs to be $L_1\ldots L_k$.  Since $\lambda$ is spiky we have $\lambda(L_1\cup\cdots\cup L_k)= (k+1) + (n-k + 1) - n = 2$. Therefore, $r_M(L_1\cup\cdots\cup L_k)+r_M^*(L_1\cup\cdots\cup L_k)=2k+2$. 

Similarly we know that $\lambda(L_1\cup\{z_2,\ldots,z_k\})=k+1$ and so since $L_1\cup\{z_2,\ldots,z_k\}$ only has $k+1$ elements we must have $r_M(L_1\cup\{z_2,\ldots,z_k\})=k+1=r_M^*(L_1\cup\{z_2,\ldots,z_k\})$ and so $r_M(L_1\cup\cdots\cup L_k)\geq k+1$ and $r_M^*(L_1\cup\cdots\cup L_k)\geq k+1$. As $r_M(L_1\cup\cdots\cup L_k)+r_M^*(L_1\cup\cdots\cup L_k)=2k+2$, it must be that $r_M(L_1\cup\cdots\cup L_k)=k+1$.

The proof that $r_M(L_1\cup\cdots\cup L_n)=n$ is similar. 

We have now shown that $M$ satisfies the definition of a spike.
\end{proof}

\begin{Corollary}\label{char}
A spiky function $\lambda$ is the connectivity function of a matroid if and only if there is some independent set $I$ of $H_n$ such that for any $X \in \mathcal T_n$ we have $\lambda(X) = n - |I \cap \{X, E_n - X\}|$. \nobreak\hfill$\square$\medskip
\end{Corollary}

This may be turned into a yet more useful characterisation. If $\lambda$ is a spiky function on $E_n$ then let $G_{\lambda}$ be the induced subgraph of $H_n$ on the vertices $X$ with $\lambda(X) = n-1$.

\begin{Lemma}\label{char2}
If $n$ is odd then every spiky function on $E_n$ is the connectivity function of a matroid. If $n$ is even then a spiky function $\lambda$ on $E_n$ is the connectivity function of a matroid if and only if there is no transversal $X$ with $X$ and $E_n - X$ in the same component of $G_{\lambda}$.
\end{Lemma}
\begin{proof}
We say that a transversal $X$ has {\em even} parity if $|\{i | x_i \in X\}|$ is even, and {\em odd} parity otherwise.

Suppose first of all that $n$ is odd. Let $\lambda$ be a spiky function on $E_n$. Let $I$ be the set of transversals $X$ such that either $\lambda(X) = n-2$ or else $X$ has even parity and $\lambda(X) = n-1$. Then $I$ cannot contain 2 transversals $X$ and $Y$ which differ in just one element; since no 2 sets of even parity differ in just one element we would have to have $\lambda(X) = n-2$ or $\lambda(Y) = n-2$, giving $\lambda(X) + \lambda(Y) < 2n-2$, which is forbidden by the definition of spiky functions. So $I$ is an independent set in $H_n$. Since for any $X$ precisely one of $X$ and $E_n - X$ has even parity, $\lambda$ is the connectivity function of the matroid $S(I)$.

Now suppose that $n$ is even, and that there is no transversal $X$ with $X$ and $E_n - X$ in the same component of $G_{\lambda}$. For any set $K$ of transversals, we write $\overline K$ for $\{E_n - X | X \in K\}$. Let $\mathcal K$ be a set of components of $G_{\lambda}$ containing precisely one of $K$ and $\overline K$ for any component $K$ of $G_{\lambda}$. Let $I$ be the set of transversals $X$ such that $\lambda(X) = 2$ or $X$ has even parity and is in $\bigcup \mathcal K$ or $X$ has odd parity and $E_n - X$ is in $\bigcup \mathcal K$. Then as in the last case $I$ is an independent set in $H_n$. Since in this case the parity of $X$ is always the same as that of $E_n - X$, once again $\lambda$ is the connectivity function of the matroid $S(I)$. 

Finally we consider the case that $n$ is even and there is some transversal $X$ with $X$ and $E_n - X$ in the same component of $G_{\lambda}$. Suppose for a contradiction that $\lambda$ is the connectivity function of a matroid, and let $I$ be as in Corollary \ref{char}. Then $|I \cap \{X, E_n - X\}| = 1$. Without loss of generality $X$ is in $I$. For any vertex $Y$ of $G_{\lambda}$ we similarly have that precisely one of $Y$ and $E-Y$ is in $I$, and for neighbouring vertices $Y$ and $Y'$ these choices must be different; we cannot have both $Y$ and $Y'$ in $I$ nor both $E_n - Y$ and $E_n - Y'$ in $I$ since $I$ is independent. Thus we may prove by induction on the distance from $X$ that any vertex $Y$ in the same component of $G_{\lambda}$ as $X$ is in $I$ if and only if it has the same parity as $X$. But since $E_n - X$ has the same parity as $X$, this implies that $E_n - X$ is also in $I$, contradicting the fact that $|I \cap \{X, E_n - X\}| = 1$.
\end{proof}

If $W$ is a set of vertices of $H_n$ which is closed under complementation then we can define a spiky function $\lambda_W$ on $E_n$ by letting $\lambda_W$ be $n-1$ on elements of $W$ and $n$ on all other transversals. Then we have $G_{\lambda_W} = H_n[W]$ and so $\lambda_W$ is the connectivity function of a matroid if and only if there is no transversal $X$ with $X$ and $E_n - X$ in the same component of $H_n[W]$. This gives a reduction of the problem of recognising whether a connectivity function is the connectivity function of a matroid to the problem of recognising whether a set $W$ of vertices of $H_n$ which is closed under complementation has an element $X$ such that $X$ and $E_n - X$ are in the same component of $H_n[W]$, which we will now exploit to show that the question of whether a connectivity function is matroidal cannot be answered positively or negatively in a polynomial number of evaluations of the connectivity function. 

Suppose that $n = 2m$ and let $V_{<m}$, $V_m$ and $V_{>m}$ be the sets of transversals $X$ such that the number of $i \leq n$ with $x_i \in X$ is, respectively, less than, equal to, or greater than $m$. Then $H_n[V_{<m} \cup V_{>m}]$ has two components, $H_n[V_{<m}]$ and $H_n[V_{>m}]$, so by Lemma \ref{char2} $\lambda(V_{<m} \cup V_{>m})$ is the connectivity function of a matroid. For any subset $A$ of $2^{E_n}$ of size less than ${n \choose m}/2$ there is some $X \in V_m$ with neither $X$ nor $E_n - X$ in $A$. But then $\lambda(V_{<m} \cup V_{>m} \cup \{X, E-X\})$ agrees with $\lambda(V_{<m} \cup V_{>m})$ on $A$ and is not the connectivity function of any matroid since $H_n[V_{<m} \cup V_{>m} \cup \{X, E-X\}]$ is connected. Since the function sending $m$ to ${2m \choose m}/2$ grows faster than any polynomial, the problem of recognising whether a connectivity function is the connectivity function of a matroid cannot be solved in a polynomial number of evaluations of the connectivity function.

Our strategy for showing that we cannot, in a polynomial number of evaluations of the connectivity function, tell when a connectivity function is not the connectivity function of a matroid, will be similar; it relies on the notion of a buffered path in $H_n$. For $X_1,X_2\subseteq E_n$ we say $X_1$ and $X_2$ are {\em neighbours} if $V(X_1)\cap V(X_2)\neq \emptyset$.

\begin{definition}
We say that a path $X_0, X_1, \ldots X_k$ is {\em buffered} if $X_k = E_n - X_0$ and the only pairs $i, j$ such that $X_i$ a neighbour of $X_j$ or $E_n - X_j$ are those with $j = i \pm 1$ or $\{i,j\} = \{1,k\}$ or $\{i, j\} = \{0, k-1\}$. We say that $k$ is the {\em length} of the buffered path.
\end{definition}

If $m$ is a natural number, we define $f(m)$ to be $\frac{2^{m+1} + 2}3$ if $m$ is odd and $\frac{2^{m+1} + 4}3$ if $m$ is even. Thus $f(1) = 2$, $f(m+1) = 2f(m)$ if $m$ is odd and $f(m+1) = 2f(m) - 2$ if $m$ is even. In particular, $f(m)$ is always a natural number, which is 2 modulo 4 if $m$ is odd and 0 modulo 4 if $m$ is even.

\begin{Lemma}
For any $m$ the graph $H_{2m}$ contains a buffered path of length $f(m)$ from $\{x_1, \ldots x_{2m}\}$ to $\{y_1, \ldots y_{2m}\}$. 
\end{Lemma}
\begin{proof}
By induction on $m$. The case $m = 1$ is trivial, since $f(1) = 2$ and any path of length 2 from $\{x_1, x_2\}$ to $\{y_1, y_2\}$ is buffered. For the induction step, suppose we have such a buffered path $X_0, \ldots X_{f(n/2)}$ in $H_n$. For $0 \leq i \leq f(m+1)$ we set
$$X'_i := \begin{cases}
X_{\frac i2} \cup \{x_{2m+1}, x_{2m+2}\}& \text{if $i$ is congruent to 0 modulo 8} \\
X_{\frac {i-1}2+1} \cup \{x_{2m+1}, x_{2m+2}\}& \text{if $i$ is congruent to 1 modulo 8} \\
X_{\frac {i-2}2 + 2} \cup \{x_{2m+1}, x_{2m+2}\}& \text{if $i$ is congruent to 2 modulo 8} \\
X_{\frac {i-3}2 + 2} \cup \{x_{2m+1}, y_{2m+2}\}& \text{if $i$ is congruent to 3 modulo 8} \\
X_{\frac {i-4}2 + 2} \cup \{y_{2m+1}, y_{2m+2}\}& \text{if $i$ is congruent to 4 modulo 8} \\
X_{\frac {i-5}2 + 3} \cup \{y_{2m+1}, y_{2m+2}\}& \text{if $i$ is congruent to 5 modulo 8} \\
X_{\frac {i-6}2 + 4} \cup \{y_{2m+1}, y_{2m+2}\}& \text{if $i$ is congruent to 6 modulo 8} \\
X_{\frac {i-7}2 + 4} \cup \{x_{2m+1}, y_{2m+2}\}& \text{if $i$ is congruent to 7 modulo 8} \\
\end{cases}$$
The only tricky part in showing that this gives a buffered path from $\{x_1, \ldots x_{2m+2}\}$ to $\{y_1, \ldots y_{2m+2}\}$ is showing that it has the correct endvertex. If $m$ is odd then $f(m)$ is 2 modulo 4 and so $f(m+1) = 2f(m)$ is 4 modulo 8. Thus $X'_{f(m+1)} = X_{\frac{2f(m) - 4}2 + 2} \cup \{y_{2m + 1}, y_{2m + 2}\} = X_{f(m)} \cup  \{y_{2m + 1}, y_{2m + 2}\} = \{y_1 \ldots y_{2m}\} \cup \{y_{2m + 1}, y_{2m + 2}\} = \{y_1, \ldots y_{2m+2}\}$. Similarly if $m$ is even then $f(m)$ is 0 modulo 4 and so $f(m+1) = 2f(m)-2$ is 6 modulo 8. Thus $X'_{f(m+1)} = X_{\frac{2f(m) -2 - 6}2 + 4} \cup \{y_{2m + 1}, y_{2m + 2}\} = X_{f(m)} \cup  \{y_{2m + 1}, y_{2m + 2}\} = \{y_1 \ldots y_{2m}\} \cup \{y_{2m + 1}, y_{2m + 2}\} = \{y_1, \ldots y_{2m+2}\}$. 
\end{proof}

Now let $X_0 \ldots X_{f(m)}$ be any buffered path from $\{x_1, \ldots x_{2m}\}$ to $\{y_1, \ldots y_{2m}\}$. Let $W$ be the set of transversals $X$ such that $X$ or $E_n - X$ appears on this path. Then $H_n[W]$ is connected so by Lemma \ref{char2} $\lambda(W)$ is not the connectivity function of any matroid. For any set $A$ of size less than $f(m) /2$, there must be some $X \in W$ such that neither $X$ nor $E_n - X$ is in $A$. But then $\lambda(W -\{X, E_n - X\})$ agrees with $\lambda(W)$ on $A$ and is the connectivity function of a matroid. This completes our proof that the problem of recognising whether a connectivity function is not the connectivity function of a matroid cannot be solved in a polynomial number of evaluations of the connectivity function.

\end{document}